\documentclass[11pt]{amsart}
\usepackage{epsfig,amsfonts,amsthm,amssymb,latexsym,amsmath}

\theoremstyle{plain}
\newtheorem{theorem}{Theorem}[section]
\newtheorem{corollary}[theorem]{Corollary}

\theoremstyle{remark}

\theoremstyle{definition}
\newtheorem{definition}[theorem]{Definition}

\title{On some Properties of Generalized Tribonacci Spinors}
\vspace{5pt}

\author{\scriptsize Gamaliel Cerda-Morales}
\date{}

\begin{document}
\maketitle

\vspace{-20pt}
\begin{center}
{\footnotesize Instituto de Matem\'aticas, Pontificia Universidad Cat\'olica de Valpara\'iso, Blanco Viel 596, Cerro Bar\'on, Valpara\'iso, Chile. \\
E-mails: gamaliel.cerda.m@mail.pucv.cl 
}\end{center}

\vspace{5pt}

\begin{abstract}
Spinors are used in physics quite extensively. The goal of this study is also the spinor structure lying in the basis of the quaternion algebra. In this paper, first, we have introduced spinors mathematically. Then, we have defined Tribonacci spinors using the generalized Tribonacci quaternions. Later, we have established the structure of algebra for these spinors. Finally, we have proved some important formulas such as Binet and Cassini-like formulas which are given for some series of numbers in mathematics for Tribonacci spinors.
\end{abstract}

\medskip
\noindent

\medskip
\noindent
\medskip

\vspace{5pt}

\section{Introduction}\label{sec:1}
\setcounter{equation}{0}
\subsection{Spinors and quaternions}\label{sec:1}

First, let us introduce the spinors mentioned by Cartan in \cite{Ca}. Consider that $\mathbf{a}=(a_{1},a_{2},a_{3})\in \mathbb{C}^{3}$ is the isotropic vector and $\mathbb{C}^{3}$ is the three-dimensional complex vector space. So, we obtain $a_{1}^{2}+a_{2}^{2}+a_{3}^{2}=0$. The set of isotropic vectors in the vector space $\mathbb{C}^{3}$ forms a two-dimensional surface in the space $\mathbb{C}^{2}$. If this two-dimensional surface is parameterized by $\mu_{1}$ and $\mu_{2}$ coordinates, then $$a_{1}=\mu_{1}^{2}-\mu_{2}^{2},\ \ a_{2}=i(\mu_{1}^{2}+\mu_{2}^{2}),\ \ a_{3}=-2\mu_{1}\mu_{2}$$ is obtained. It is seen from the solution of this equation that $$\mu_{1}=\pm \sqrt{\frac{a_{1}-ia_{2}}{2}},\ \mu_{2}=\pm \sqrt{\frac{-a_{1}-ia_{2}}{2}}.$$ It is seen that in the complex vector space $\mathbb{C}^{3}$, each isotropic vector corresponds to two vectors, $(\mu_{1},\mu_{2})$ and $(-\mu_{1},-\mu_{2})$ in the space $\mathbb{C}^{2}$. Conversely, both vectors so given in space $\mathbb{C}^{2}$ correspond to the same isotropic vector $\mathbf{a}$. Cartan stated that the two-dimensional complex vectors
\begin{equation}\label{eq:1}
\mu=(\mu_{1},\mu_{2})  \cong \left[ \begin{array}{c} \mu_{1} \\ \mu_{2} \end{array}\right]
\end{equation}
described in this way are called as $\emph{spinor}$ (see \cite{Ca}). The set of spinors is denoted by $\mathrm{S}_{p}$. In addition, Cartan emphasized that spinors are not only two-dimensional complex vectors, but also represent three-dimensional complex isotropic vectors (see also, for example, \cite{Sa}).

Let us consider that the complex conjugate of spinor $\mu$ is $\overline{\mu}$. So, Cartan in \cite{Ca} wrote that the conjugate of spinor $\mu$ is
\begin{equation}\label{eq:2}
\widetilde{\mu}=i\mathrm{C}\overline{\mu} =i\left[ \begin{array}{cc} 0&1 \\ -1 &0 \end{array}\right]\left[ \begin{array}{c} \overline{\mu_{1}} \\ \overline{\mu_{2}} \end{array}\right]=\left[ \begin{array}{c} i\overline{\mu_{2} }\\ -i\overline{\mu_{1}} \end{array}\right],
\end{equation}
where $\mathrm{C}= \left[ \begin{array}{cc} 0& 1\\ -1 &0 \end{array}\right]$. Moreover, Torres Del Castillo and S\'anchez Barrales \cite{To1} wrote that the mate of spinor $\mu$ is
\begin{equation}\label{eq:3}
\widehat{\mu}=-\mathrm{C}\overline{\mu}=- \left[ \begin{array}{cc} 0& 1\\ -1 &0 \end{array}\right]\left[ \begin{array}{c} \overline{\mu_{1}} \\ \overline{\mu_{2}} \end{array}\right]=\left[ \begin{array}{c} -\overline{\mu_{2} }\\ \overline{\mu_{1}} \end{array}\right].
\end{equation}

Furthermore, a real quaternion is defined with $$\mathrm{q}=q_{0}+\sum_{l=1}^{3}q_{l}\mathrm{e}_{l},\ \  q_{0}\mathrm{e}_{0}=q_{0},$$ where $q_{l}\in \mathbb{R}$ ($l=0,1,2,3$) and the quaternion basis $\{\mathrm{e}_{l}:\ l=0,1,2,3\}$ is given such that $\mathrm{e}_{1}^{2}=\mathrm{e}_{2}^{2}=\mathrm{e}_{3}^{2}=-\mathrm{e}_{0}$, $\mathrm{e}_{1}\mathrm{e}_{2}=\mathrm{e}_{3}=-\mathrm{e}_{2}\mathrm{e}_{1}$, $\mathrm{e}_{2}\mathrm{e}_{3}=\mathrm{e}_{1}=-\mathrm{e}_{3}\mathrm{e}_{2}$ and $\mathrm{e}_{3}\mathrm{e}_{1}=\mathrm{e}_{2}=-\mathrm{e}_{1}\mathrm{e}_{3}$.

Let  $S_{\mathrm{q}}=q_{0}$ and $V_{\mathrm{q}}=\sum_{l=1}^{3}q_{l}\mathrm{e}_{l}$ be scalar and vectorial parts of the quaternion $\mathrm{q}$. So, we can write the quaternion $\mathrm{q}$ as $\mathrm{q}=S_{\mathrm{\mathbf{q}}}+V_{\mathrm{\mathbf{q}}}$. The set of these quaternions is $\mathbb{H}$. Let  $\mathrm{p}=S_{\mathrm{p}}+V_{\mathrm{p}},  \mathrm{q}=S_{\mathrm{q}}+V_{\mathrm{q}} \in \mathbb{H}$ be two real quaternions. So, the quaternion product of these quaternions
$$
\mathrm{p}\times \mathrm{q}=S_{\mathrm{p}}S_{\mathrm{q}}-\langle V_{\mathrm{p}},V_{\mathrm{q}}\rangle +S_{\mathrm{p}}V_{\mathrm{q}}+S_{\mathrm{q}}V_{\mathrm{p}}+V_{\mathrm{p}}\wedge V_{\mathrm{q}},
$$ where $\langle,\rangle$ is usual inner product and $\wedge $ is cross product in real vector space $\mathbb{R}^{3}$. It is clear that the product of two real quaternions is non-commutative.In addition, if we consider that $\mathrm{q}^{*}$ is the conjugate of the quaternion $\mathrm{q}$, $\mathrm{q}^{*}$ is equal to $\mathrm{q}^{*}=S_{\mathrm{q}}-V_{\mathrm{q}}$, the norm of a quaternion $\mathrm{q}$ is
$Nr(\mathrm{q})=\sum_{l=0}^{3}q_{l}^{2}$. If $Nr(\mathrm{q})=1$, then $\mathrm{q}$ is called a unit quaternion.

Now, we give one relationship between spinors and quaternions expressed by Vivarelli (see \cite{Vi}). This correspondence between any quaternion $\mathrm{q}=q_{0}+\sum_{l=1}^{3}q_{l}\mathrm{e}_{l}$ and a spinor $\mu=\left[ \begin{array}{c} \mu_{1} \\ \mu_{2} \end{array}\right]$ is
\begin{equation}\label{eq:4}
\begin{aligned}
\sigma:\ &\mathbb{H}  \longrightarrow \mathrm{S}_{p}\\
&\ \ \mathrm{q} \longmapsto \sigma\left( q_{0}+\sum_{l=1}^{3}q_{l}\mathrm{e}_{l}\right)=\left[ \begin{array}{c} q_{3}+iq_{0} \\ q_{1}+iq_{2} \end{array}\right].
\end{aligned}
\end{equation}
(see \cite[Eq. (23)]{Vi}). So, we obtain a spinor formulation of the kinematics of rotation, which extends the quaternion formulation. Moreover, this function $\sigma$ is clearly linear and injective. Then, we can write $\sigma(\mathrm{p}+\mathrm{q})=\sigma(\mathrm{p})+\sigma(\mathrm{q})$, $\sigma(k\mathrm{q})=k \sigma(\mathrm{q})$, $k \in \mathbb{R}$ and $\ker(\sigma)=\{0\}$. If we take that the conjugate of the quaternion $\mathrm{q}$ is $\mathrm{q}^{*}=q_{0}-\sum_{l=1}^{3}q_{l}\mathrm{e}_{l}$, the spinor corresponding to conjugate quaternion can be written as $$\sigma(\mathrm{q}^{*})=\sigma\left( q_{0}-\sum_{l=1}^{3}q_{l}\mathrm{e}_{l}\right)=\left[ \begin{array}{c} -q_{3}+iq_{0} \\ -q_{1}-iq_{2} \end{array}\right].$$
Vivarelli \cite[Eq. (25)]{Vi} associated to the product of two quaternions $\mathrm{p}\times \mathrm{q}$, which is equal to a spinor matrix product as follows:
\begin{equation}\label{prod}
\mathrm{p}\times \mathrm{q}\longrightarrow  -i\breve{P}Q,
\end{equation}
where $$\breve{P}=\left[ \begin{array}{cc} p_{3}+ip_{0} & p_{1}-ip_{2}\\ p_{1}+ip_{2}& -p_{3}+ip_{0} \end{array}\right],\  Q=\left[ \begin{array}{c} q_{3}+iq_{0} \\ q_{1}+iq_{2} \end{array}\right].$$

There are four known ways of approaching the theory of spinors in the Euclidean three-space: by Cartan's isotropic vectors, by Clifford algebra as exemplified by Pauli matrices, by spinor ring algebra, by stereographic projection (see \cite{Ca,Sa,To}).

\subsection{Generalized Tribonacci quaternions}
The definitions of generalized Tribonacci quaternions were first given by Cerda-Morales \cite{Ce1}. So, generalized Tribonacci can be written as $$Q_{v,n}=V_{n}+\sum_{l=1}^{3}V_{n+l}\mathrm{e}_{l},\ \ V_{n}\mathrm{e}_{0}=V_{n}$$ where $V_{n}=V_{n}(V_{0},V_{1},V_{2};r,s,t)$ is the $n$-th generalized Tribonacci number defined by
\begin{equation}\label{eq:5}
V_{n}=rV_{n-1}+sV_{n-2}+tV_{n-3},\ n\geq 3,
\end{equation}
and $V_{0}$, $V_{1}$, $V_{2}$ are arbitrary integers and $r$, $s$ and $t$ are real numbers. Generalized Tribonacci quaternions are the third-order linear recurrence sequence and for $n\geq 0$ one can write
\begin{equation}\label{eq:6}
Q_{v,n+3}=rQ_{v,n+2}+sQ_{v,n+1}+tQ_{v,n}.
\end{equation}
Let $Q_{v,n}$ be $n$-th generalized Tribonacci quaternion. Then, the Binet formula for this quaternion is 
\begin{equation}\label{eq:66}
Q_{v,n}=\frac{\mathrm{P}\underline{\alpha}\alpha^{n}}{(\alpha-\omega_{1})(\alpha-\omega_{2})}-\frac{\mathrm{Q}\underline{\omega_{1}}\omega_{1}^{n}}{(\alpha-\omega_{1})(\omega_{1}-\omega_{2})}+\frac{\mathrm{R}\underline{\omega_{2}}\omega_{2}^{n}}{(\alpha-\omega_{2})(\omega_{1}-\omega_{2})},
\end{equation}
where $\mathrm{P}=V_{2}-(\omega_{1}+\omega_{2})V_{1}+\omega_{1}\omega_{2}V_{0}$, $\mathrm{Q}=V_{2}-(\alpha+\omega_{2})V_{1}+\alpha\omega_{2}V_{0}$, $\mathrm{R}=V_{2}-(\alpha+\omega_{1})V_{1}+\alpha\omega_{1}V_{0}$, $\underline{\alpha}=\mathrm{e}_{0}+\sum_{l=1}^{3}\alpha^{l}\mathrm{e}_{l}$, $\underline{\omega_{1}}=\mathrm{e}_{0}+\sum_{l=1}^{3}\omega_{1}^{l}\mathrm{e}_{l}$, $\underline{\omega_{2}}=\mathrm{e}_{0}+\sum_{l=1}^{3}\omega_{2}^{l}\mathrm{e}_{l}$ and $\alpha$, $\omega_{1}$ and $\omega_{2}$ are the roots of the cubic equation $x^{3}-rx^{2}-sx-t=0$ (see \cite{Ce1}). 

In addition, for the generalized Tribonacci quaternion $Q_{v,n}$, the generating function is 
$$
g_{Q_{v,n}}(x)=\frac{Q_{v,0}+(Q_{v,1}-rQ_{v,0})x+(Q_{v,2}-rQ_{v,1}-sQ_{v,0})x^{2}}{1-rx-sx^{2}-tx^{3}}.
$$
Furthermore, the summation formula for generalized Tribonacci quaternions is as follows: $$\sum_{l=0}^{n}Q_{v,l}=\frac{1}{\delta}\left(Q_{v,n+2}+(1-r)Q_{v,n+1}+tQ_{v,n}+\omega \right),$$ where $\delta=r+s+t-1$, $\omega=\lambda+\mathrm{e}_{1}(\lambda-\delta V_{0})+\mathrm{e}_{2}(\lambda-\delta(V_{0}+V_{1}))+\mathrm{e}_{3}(\lambda-\delta(V_{0}+V_{1}+V_{2}))$ and $\lambda=(r+s-1)V_{0}+(r-1)V_{1}-V_{2}$. Moreover, Cerda-Morales gave the generalized Tribonacci quaternion matrix as follows:
\begin{equation}\label{eq:7}
\mathbf{Q}_{v}=\left[\begin{array}{ccc}
Q_{v,4}&K_{v,2}&tQ_{v,3}\\
Q_{v,3}&K_{v,1}&tQ_{v,2}\\
Q_{v,2}&K_{v,0}&tQ_{v,1}
\end{array}\right],
\end{equation}
where $K_{v,n}=sQ_{v,n+1}+tQ_{v,n}$. With the help of this matrix, the author obtained the next formula for generalized Tribonacci quaternions $$Q_{v,n+2}=Q_{v,2}U_{n+2}+(sQ_{v,1}+tQ_{v,0})U_{n+1}+tQ_{v,1}U_{n},\ n\geq 0$$ where $U_{n}=V_{n}(0,0,1;r,s,t)$.

\section{Main Theorems and Proofs}
In this section, we consider that there is a spinor corresponding to every real quaternion and we match each generalized Tribonacci quaternion with one spinor which has two complex components. Then, we introduce generalized Tribonacci spinors. Similarly, we express the particular case of third-order Jacobsthal spinors. Then, using the recent work of Eri\c{s}ar and G\"ung\"or in \cite{Er}, we give some theorems and formulas for generalized Tribonacci and Tribonacci spinors.

Let the quaternion $Q_{v,n}=V_{n}+\sum_{l=1}^{3}V_{n+l}\mathrm{e}_{l}$ be $n$-th generalized Tribonacci quaternion where $V_{n}$ is $n$-th generalized Tribonacci number. Let the set of generalized Tribonacci quaternions be $\mathbb{T}_{r,s,t}$.

Now, we consider the correspondence between spinors and quaternions in Eq. (\ref{eq:4}) and we write the following transformation:
\begin{equation}\label{eq:tri}
\begin{aligned}
\sigma:\ &\mathbb{T}_{r,s,t}  \longrightarrow \mathrm{S}_{p}\\
&\ \ \mathrm{q} \longmapsto \sigma\left( V_{n}+\sum_{l=1}^{3}V_{n+l}\mathrm{e}_{l}\right)=\left[ \begin{array}{c} V_{n+3}+iV_{n} \\ V_{n+1}+iV_{n+2} \end{array}\right]\equiv \mathrm{A}_{v,n},
\end{aligned}
\end{equation}
where $\{\mathrm{e}_{l}:\ l=1,2,3\}$ coincide with basis vectors given for real quaternions. This transformation is linear and injective but it is not surjective. So, with the help of this transformation, we obtain a new sequence from generalized Tribonacci quaternions. We say that this new sequence is a generalized Tribonacci spinor sequence. Note that 
\begin{equation}\label{eq:8}
\mathrm{A}_{v,n+3}=r\mathrm{A}_{v,n+2}+s\mathrm{A}_{v,n+1}+t\mathrm{A}_{v,n},\ n\geq 0,
\end{equation}
can be written. Then, the generalized Tribonacci spinor sequence is the linear recurrence sequence.
Now, we give some algebraic definitions for generalized Tribonacci spinors.

\begin{definition}[Conjugates]
Let the conjugate of the generalized Tribonacci quaternion $Q_{v,n}$ be $Q_{v,n}^{*}=V_{n}-\sum_{l=1}^{3}V_{n+l}\mathrm{e}_{l}$. So, from the correspondence in Eq. (\ref{eq:8}), the generalized Tribonacci spinor $\mathrm{A}_{v,n}^{*}$ corresponding to conjugate of generalized Tribonacci quaternion is written by $$\sigma\left( V_{n}-\sum_{l=1}^{3}V_{n+l}\mathrm{e}_{l}\right)=\left[ \begin{array}{c} -V_{n+3}+iV_{n} \\ -V_{n+1}-iV_{n+2} \end{array}\right]\equiv \mathrm{A}_{v,n}^{*}.$$

On the other hand, we can write that the ordinary complex conjugate of generalized Tribonacci spinor $\mathrm{A}_{v,n}$ is $$\mathrm{A}_{v,n}=\left[ \begin{array}{c} V_{n+3}+iV_{n} \\ V_{n+1}+iV_{n+2} \end{array}\right] \longrightarrow \overline{\mathrm{A}}_{v,n}=\left[ \begin{array}{c} V_{n+3}-iV_{n} \\ V_{n+1}-iV_{n+2} \end{array}\right].$$ 

According to the spinor conjugate given in Eq. (\ref{eq:2}) by Cartan, generalized Tribonacci spinor conjugate $\widetilde{\mathrm{A}}_{v,n}=i\mathrm{C}\overline{\mathrm{A}}_{v,n}$ is as follows:
 $$\mathrm{A}_{v,n}=\left[ \begin{array}{c} V_{n+3}+iV_{n} \\ V_{n+1}+iV_{n+2} \end{array}\right] \longrightarrow \widetilde{\mathrm{A}}_{v,n}=\left[ \begin{array}{c} V_{n+2}+iV_{n+1} \\ -V_{n}-iV_{n+3} \end{array}\right].$$
 
In addition to that, Torres Del Castillo and S\'anchez Barrales obtained the mate of spinor in Eq. (\ref{eq:3}). According to this, the mate of generalized Tribonacci spinor is $\widehat{\mathrm{A}}_{v,n}=-\mathrm{C}\overline{\mathrm{A}}_{v,n}$ and
 $$\mathrm{A}_{v,n}=\left[ \begin{array}{c} V_{n+3}+iV_{n} \\ V_{n+1}+iV_{n+2} \end{array}\right] \longrightarrow \widehat{\mathrm{A}}_{v,n}=\left[ \begin{array}{c} -V_{n+1}+iV_{n+2} \\ V_{n+3}-iV_{n} \end{array}\right].$$
\end{definition}

The following corollary gives the relationship between the conjugates for generalized Tribonacci spinors.

\begin{corollary}\label{cor:1}
Let the $n$-th generalized Tribonacci spinor be $\mathrm{A}_{v,n}$. The correspondences between
conjugates of generalized Tribonacci spinors hold
\begin{equation}\label{cor:1}
\mathrm{C}\widehat{\mathrm{A}}_{v,n}= \left[ \begin{array}{cc} 0& 1\\ -1 &0 \end{array}\right]\left[ \begin{array}{c} -V_{n+1}+iV_{n+2} \\ V_{n+3}-iV_{n} \end{array}\right]=\overline{\mathrm{A}}_{v,n},
\end{equation}
\begin{equation}\label{cor:2}
i\widetilde{\mathrm{A}}_{v,n}= i\left[ \begin{array}{c} V_{n+2}+iV_{n+1} \\ -V_{n}-iV_{n+3} \end{array}\right]=\widehat{\mathrm{A}}_{v,n},
\end{equation}
\begin{equation}\label{cor:3}
i\mathrm{C}\widetilde{\mathrm{A}}_{v,n}= \left[ \begin{array}{cc} 0& i\\ -i &0 \end{array}\right]\left[ \begin{array}{c} V_{n+2}+iV_{n+1} \\ -V_{n}-iV_{n+3} \end{array}\right]=\overline{\mathrm{A}}_{v,n}.
\end{equation}
\end{corollary}

\begin{definition}[Norm]
The norm of a generalized Tribonacci quaternion $Nr(Q_{v,n})=Q_{v,n}Q_{v,n}^{*}$ is equal to the norm of associated generalized Tribonacci spinor $$Nr(Q_{v,n})=\left[ \begin{array}{cc} V_{n+3}-iV_{n} & V_{n+1}-iV_{n+2} \end{array}\right]\left[ \begin{array}{c} V_{n+3}+iV_{n} \\ V_{n+1}+iV_{n+2} \end{array}\right].$$

In addition, if we use Corollary \ref{cor:1}, then we give the norm of generalized tribonacci spinor $$Nr(Q_{v,n})=\overline{\mathrm{A}}_{v,n}^{t}\mathrm{A}_{v,n}=-\widehat{\mathrm{A}}_{v,n}^{t}\mathrm{C}^{t}\mathrm{A}_{v,n}=-i\widetilde{\mathrm{A}}_{v,n}^{t}\mathrm{C}^{t}\mathrm{A}_{v,n}.$$
\end{definition}

Now, we give the Binet formula for generalized Tribonacci spinors solving generalized Tribonacci spinor recurrence relation in Eq. (\ref{eq:8}).
\begin{theorem}\label{teo:1}
For $n\geq 0$, the Binet formula for the $n$-th generalized Tribonacci spinor $\mathrm{A}_{v,n}$ is that
\begin{equation}\label{eq:10}
\begin{aligned}
\mathrm{A}_{v,n}&=\frac{\mathrm{P}\alpha^{n}}{(\alpha-\omega_{1})(\alpha-\omega_{2})}\left[ \begin{array}{c} \alpha^{3}+i \\ \alpha+i\alpha^{2} \end{array}\right]-\frac{\mathrm{Q}\omega_{1}^{n}}{(\alpha-\omega_{1})(\omega_{1}-\omega_{2})}\left[ \begin{array}{c} \omega_{1}^{3}+i \\ \omega_{1}+i\omega_{1}^{2} \end{array}\right]\\
&\ \ + \frac{\mathrm{R}\omega_{2}^{n}}{(\alpha-\omega_{2})(\omega_{1}-\omega_{2})}\left[ \begin{array}{c} \omega_{2}^{3}+i \\ \omega_{2}+i\omega_{2}^{2} \end{array}\right],
\end{aligned}
\end{equation}
where $\mathrm{P}$, $\mathrm{Q}$ and $\mathrm{R}$ as in Eq. (\ref{eq:66}).
\end{theorem}
\begin{proof}
We consider the generalized Tribonacci spinor sequence $\{\mathrm{A}_{v,n}\}_{n\geq 0}$. The characteristic equation of recurrence relation of generalized spinors is also $x^{3}-rx^{2}-sx-t=0$. Moreover, the roots of this equation are $\alpha$, $\omega_{1}$ and $\omega_{2}$ as in Eq. (\ref{eq:66}). We assume the initial values $$\mathrm{A}_{v,0}=\left[ \begin{array}{c} V_{3}+iV_{0} \\ V_{1}+iV_{2} \end{array}\right],\ \ \mathrm{A}_{v,1}=\left[ \begin{array}{c} V_{4}+iV_{1} \\ V_{2}+iV_{3} \end{array}\right],\ \ \mathrm{A}_{v,2}=\left[ \begin{array}{c} V_{5}+iV_{2} \\ V_{3}+iV_{4} \end{array}\right]$$ for the equation $\mathrm{A}_{v,n}=a_{1}\alpha^{n}+a_{2}\omega_{1}^{n}+a_{3}\omega_{2}^{n}$. So, for generalized Tribonacci spinor $\mathrm{A}_{v,0}$, ($n=0$) we have $\mathrm{A}_{v,0}=a_{1}+a_{2}+a_{3}=\left[ \begin{array}{c} V_{3}+iV_{0} \\ V_{1}+iV_{2} \end{array}\right]$, for generalized Tribonacci spinor $\mathrm{A}_{v,1}$, ($n=1$) also we obtain $\mathrm{A}_{v,1}=a_{1}\alpha+a_{2}\omega_{1}+a_{3}\omega_{2}=\left[ \begin{array}{c} V_{4}+iV_{1} \\ V_{2}+iV_{3} \end{array}\right]$ and finally for $n=2$, we have  $\mathrm{A}_{v,2}=a_{1}\alpha^{2}+a_{2}\omega_{1}^{2}+a_{3}\omega_{2}^{2}=\left[ \begin{array}{c} V_{5}+iV_{2} \\ V_{3}+iV_{4} \end{array}\right]$. If we make necessary adjustments we find that the spinors $$a_{1}=\frac{\mathrm{P}}{(\alpha-\omega_{1})(\alpha-\omega_{2})}\left[ \begin{array}{c} \alpha^{3}+i \\ \alpha+i\alpha^{2} \end{array}\right],$$ $$a_{2}=-\frac{\mathrm{Q}}{(\alpha-\omega_{1})(\omega_{1}-\omega_{2})}\left[ \begin{array}{c} \omega_{1}^{3}+i \\ \omega_{1}+i\omega_{1}^{2} \end{array}\right]$$ and $$a_{3}=\frac{\mathrm{R}}{(\alpha-\omega_{2})(\omega_{1}-\omega_{2})}\left[ \begin{array}{c} \omega_{2}^{3}+i \\ \omega_{2}+i\omega_{2}^{2} \end{array}\right].$$
So, we obtain that Binet's formula for $\mathrm{A}_{v,n}$ is obtained as follows:
\begin{align*}
\mathrm{A}_{v,n}&=a_{1}\alpha^{n}+a_{2}\omega_{1}^{n}+a_{3}\omega_{2}^{n}\\
&=\frac{\mathrm{P}\alpha^{n}}{(\alpha-\omega_{1})(\alpha-\omega_{2})}\left[ \begin{array}{c} \alpha^{3}+i \\ \alpha+i\alpha^{2} \end{array}\right]-\frac{\mathrm{Q}\omega_{1}^{n}}{(\alpha-\omega_{1})(\omega_{1}-\omega_{2})}\left[ \begin{array}{c} \omega_{1}^{3}+i \\ \omega_{1}+i\omega_{1}^{2} \end{array}\right]\\
&\ \ + \frac{\mathrm{R}\omega_{2}^{n}}{(\alpha-\omega_{2})(\omega_{1}-\omega_{2})}\left[ \begin{array}{c} \omega_{2}^{3}+i \\ \omega_{2}+i\omega_{2}^{2} \end{array}\right].
\end{align*}
The result is obtained.
\end{proof}

\begin{corollary}
For $n\geq 0$, the Binet formula for the $n$-th Tribonacci spinor $\mathrm{T}_{n}$ is that
\begin{equation}\label{eq:10}
\begin{aligned}
\mathrm{T}_{n}&=\frac{\alpha^{n+1}}{(\alpha-\omega_{1})(\alpha-\omega_{2})}\left[ \begin{array}{c} \alpha^{3}+i \\ \alpha+i\alpha^{2} \end{array}\right]-\frac{\omega_{1}^{n+1}}{(\alpha-\omega_{1})(\omega_{1}-\omega_{2})}\left[ \begin{array}{c} \omega_{1}^{3}+i \\ \omega_{1}+i\omega_{1}^{2} \end{array}\right]\\
&\ \ + \frac{\omega_{2}^{n+1}}{(\alpha-\omega_{2})(\omega_{1}-\omega_{2})}\left[ \begin{array}{c} \omega_{2}^{3}+i \\ \omega_{2}+i\omega_{2}^{2} \end{array}\right],
\end{aligned}
\end{equation}
where $\mathrm{T}_{n}=\mathrm{A}_{v,n}$ and $V_{n}=V_{n}(0,1,1;1,1,1)$ is the $n$-th classic Tribonacci number.
\end{corollary}

Now, we obtain the generating functions for generalized Tribonacci spinors.
\begin{theorem}\label{teo:3}
For $n\geq 0$, the generating function for the $n$-th generalized Tribonacci spinor $\mathrm{A}_{v,n}$ is
\begin{equation}\label{eq:12}
g_{\mathrm{A}_{v,n}}(x)=\frac{1}{\rho(x)}\left[ \begin{array}{c} \rho_{1}V_{3}+\rho_{2}V_{4}+x^{2}V_{5}+i (\rho_{1}V_{0}+\rho_{2}V_{1}+x^{2}V_{2})\\ \rho_{1}V_{1}+\rho_{2}V_{2}+x^{2}V_{3}+i (\rho_{1}V_{2}+\rho_{2}V_{3}+x^{2}V_{4}) \end{array}\right],
\end{equation}
where $\rho(x)=1-rx-sx^{2}-tx^{3}$, $\rho_{1}=\rho_{1}(x)=1-rx-sx^{2}$ and $\rho_{2}=\rho_{2}(x)=x-rx^{2}$. Furthermore, $V_{n}$ as in Eq. (\ref{eq:5}).
\end{theorem}
\begin{proof}
If the generating function $g_{\mathrm{A}_{v,n}}(x)=\sum_{n\geq 0}\mathrm{A}_{v,n}x^{n}$ is considered, then using the equations $rxg_{\mathrm{A}_{v,n}}(x)$, $sx^{2}g_{\mathrm{A}_{v,n}}(x)$ and $tx^{3}g_{\mathrm{A}_{v,n}}(x)$, we obtain that
\begin{align*}
(1-rx-&sx^{2}-tx^{3})g_{\mathrm{A}_{v,n}}(x)\\
&=\mathrm{A}_{v,0}+(\mathrm{A}_{v,1}-r\mathrm{A}_{v,0})x+(\mathrm{A}_{v,2}-r\mathrm{A}_{v,1}-t\mathrm{A}_{v,0})x^{2}\\
&=\left[ \begin{array}{c} V_{3}+i V_{0}\\ V_{1}+i V_{2} \end{array}\right]+\left[ \begin{array}{c} V_{4}-rV_{3}+i (V_{1}-rV_{0})\\ V_{2}-rV_{1}+i (V_{3}-rV_{2}) \end{array}\right]x\\
&\ \ + \left[ \begin{array}{c} V_{5}-rV_{4}-sV_{3}+i (V_{2}-rV_{1}-sV_{0})\\ V_{3}-rV_{2}-sV_{1}+i (V_{4}-rV_{3}-sV_{2}) \end{array}\right]x^{2}.
\end{align*}
If we regulate the last equation, we have
\begin{align*}
(1-&rx-sx^{2}-tx^{3})g_{\mathrm{A}_{v,n}}(x)\\
&= \left[ \begin{array}{c} \rho_{1}V_{3}+\rho_{2}V_{4}+x^{2}V_{5}+i (\rho_{1}V_{0}+\rho_{2}V_{1}+x^{2}V_{2})\\ \rho_{1}V_{1}+\rho_{2}V_{2}+x^{2}V_{3}+i (\rho_{1}V_{2}+\rho_{2}V_{3}+x^{2}V_{4}) \end{array}\right],
\end{align*}
where $\rho_{1}=\rho_{1}(x)=1-rx-sx^{2}$ and $\rho_{2}=\rho_{2}(x)=x-rx^{2}$. Then, the result is obtained.
\end{proof}

\begin{corollary}
For $n\geq 0$, the generating function for the $n$-th Tribonacci spinor $\mathrm{T}_{n}$ is
\begin{equation}\label{eq:12}
g_{\mathrm{T}_{n}}(x)=\frac{1}{1-x-x^{2}-x^{3}}\left[ \begin{array}{c} 2+2x+x^{2}+i x\\ 1+i (1+x+x^{2}) \end{array}\right].
\end{equation}
\end{corollary}

Now, similar to $\mathbf{Q}_{v}$-matrix given for generalized Tribonacci quaternions in Eq. (\ref{eq:7}) we can also give one matrix for generalized Tribonacci spinors. So, with the aid of this matrix, we can obtain am special formula for generalized Tribonacci spinors. Thus, we first express the following theorem.

\begin{theorem}\label{teo:4}
Generalized Tribonacci spinor matrix, which has the same behavior as the $\mathbf{Q}_{v}$-matrix given for generalized Tribonacci quaternions, is given by $$\mathcal{Q}_{v}=-\left[\begin{array}{ccc}
\mathrm{A}_{v,4}&\breve{\mathrm{K}}_{v,2}&t\breve{\mathrm{A}}_{v,3}\\
\mathrm{A}_{v,3}&\breve{\mathrm{K}}_{v,1}&t\breve{\mathrm{A}}_{v,2}\\
\mathrm{A}_{v,2}&\breve{\mathrm{K}}_{v,0}&t\breve{\mathrm{A}}_{v,1}
\end{array}\right],\ \breve{\mathrm{K}}_{v,n}=s\breve{\mathrm{A}}_{v,n+1}+t\breve{\mathrm{A}}_{v,n}$$
where $\breve{\mathrm{A}}_{v,0}=\left[ \begin{array}{cc} V_{3}+iV_{0} &V_{1}-iV_{2} \\ V_{1}+iV_{2} &-V_{3}+iV_{0} \end{array}\right]$, $\breve{\mathrm{A}}_{v,1}=\left[ \begin{array}{cc} V_{4}+iV_{1}&V_{2}-iV_{3}  \\ V_{2}+iV_{3} &-V_{4}+iV_{1}\end{array}\right]$ and $\breve{\mathrm{A}}_{v,2}=\left[ \begin{array}{cc} V_{5}+iV_{2}&V_{3}-iV_{4} \\ V_{3}+iV_{4} &-V_{5}+iV_{2}\end{array}\right]$.
\end{theorem}
\begin{proof}
Let $n$-th generalized Tribonacci spinor $\mathrm{A}_{v,n}$ correspond to $n$-th generalized Tribonacci quaternion $Q_{v,n}$. Let us consider the $\mathbf{Q}_{v}$-matrix in Eq. (\ref{eq:7}); then the product of two quaternions can be written by spinor matrix in Eq. (\ref{prod}). For example we have the following equations $Q_{v,1}K_{v,3}Q_{v,4}=-\breve{\mathrm{A}}_{v,1}\breve{\mathrm{K}}_{v,1}\mathrm{A}_{v,4}$ ($Q_{v,1}K_{v,1}Q_{v,4}$ is preferred over the multiplication $Q_{v,4}K_{v,1}Q_{v,1}$), where $\breve{\mathrm{K}}_{v,1}=s\breve{\mathrm{A}}_{v,2}+t\breve{\mathrm{A}}_{v,1}$. Moreover, for $\mathbf{Q}_{v}$-matrix we obtain that
\begin{align*}
\frac{1}{t}\det(\mathbf{Q}_{v})&=Q_{v,1} \begin{vmatrix} Q_{v,4}  & K_{v,2}  \\ Q_{v,3} & K_{v,1} \end{vmatrix}-Q_{v,2} \begin{vmatrix} Q_{v,4}  & K_{v,2}  \\ Q_{v,2} & K_{v,0} \end{vmatrix}+Q_{v,3} \begin{vmatrix} Q_{v,3}  & K_{v,1}  \\ Q_{v,2} & K_{v,0} \end{vmatrix}\\
&=Q_{v,1}K_{v,1}Q_{v,4}+Q_{v,2}K_{v,2}Q_{v,2}+Q_{v,3}K_{v,0}Q_{v,3}\\
&\ \ -Q_{v,1}K_{v,2}Q_{v,3}-Q_{v,2}K_{v,0}Q_{v,4}-Q_{v,3}K_{v,1}Q_{v,2}\\
&=-\breve{\mathrm{A}}_{v,1}\breve{\mathrm{K}}_{v,1}\mathrm{A}_{v,4}-\breve{\mathrm{A}}_{v,2}\breve{\mathrm{K}}_{v,2}\mathrm{A}_{v,2}-\breve{\mathrm{A}}_{v,3}\breve{\mathrm{K}}_{v,0}\mathrm{A}_{v,3}\\
&\ \ + \breve{\mathrm{A}}_{v,1}\breve{\mathrm{K}}_{v,2}\mathrm{A}_{v,3}+\breve{\mathrm{A}}_{v,2}\breve{\mathrm{K}}_{v,0}\mathrm{A}_{v,4}+\breve{\mathrm{A}}_{v,3}\breve{\mathrm{K}}_{v,1}\mathrm{A}_{v,2}.
\end{align*}
Finally, we can choose $$\mathcal{Q}_{v}=-\left[\begin{array}{ccc}
\mathrm{A}_{v,4}&\breve{\mathrm{K}}_{v,2}&t\breve{\mathrm{A}}_{v,3}\\
\mathrm{A}_{v,3}&\breve{\mathrm{K}}_{v,1}&t\breve{\mathrm{A}}_{v,2}\\
\mathrm{A}_{v,2}&\breve{\mathrm{K}}_{v,0}&t\breve{\mathrm{A}}_{v,1}
\end{array}\right],$$
where $\breve{\mathrm{A}}_{v,0}=\left[ \begin{array}{cc} V_{3}+iV_{0} &V_{1}-iV_{2} \\ V_{1}+iV_{2} &-V_{3}+iV_{0} \end{array}\right]$, $\breve{\mathrm{A}}_{v,1}=\left[ \begin{array}{cc} V_{4}+iV_{1}&V_{2}-iV_{3}  \\ V_{2}+iV_{3} &-V_{4}+iV_{1}\end{array}\right]$ and $\breve{\mathrm{A}}_{v,2}=\left[ \begin{array}{cc} V_{5}+iV_{2}&V_{3}-iV_{4} \\ V_{3}+iV_{4} &-V_{5}+iV_{2}\end{array}\right]$. 
\end{proof}

\begin{corollary}
Tribonacci spinor matrix, which has the same behavior as the $\mathbf{T}_{v}$-matrix given for Tribonacci quaternions, is given by $$\mathcal{T}=-\left[\begin{array}{ccc}
\mathrm{T}_{4}&\breve{\mathrm{T}}_{3}+\breve{\mathrm{T}}_{2}&\breve{\mathrm{T}}_{3}\\
\mathrm{T}_{3}&\breve{\mathrm{T}}_{2}+\breve{\mathrm{T}}_{1}&\breve{\mathrm{T}}_{2}\\
\mathrm{T}_{2}&\breve{\mathrm{T}}_{1}+\breve{\mathrm{T}}_{0}&\breve{\mathrm{T}}_{1}
\end{array}\right],$$
where $$\breve{\mathrm{T}}_{0}=\left[ \begin{array}{cc} 2 &1-i \\ 1+i &-2 \end{array}\right], \breve{\mathrm{T}}_{1}=\left[ \begin{array}{cc} 4+i&1-2i  \\ 1+2i &-4+i\end{array}\right], \breve{\mathrm{T}}_{2}=\left[ \begin{array}{cc} 7+i&2-4i \\ 2+4i &-7+i\end{array}\right].$$
\end{corollary}

Now, we give the determinant formula for Tribonacci spinors as a result of Theorem \ref{teo:4} and \cite[Theorem 3.1]{Ce1}:
$$\mathbf{Q}_{v}\left[\begin{array}{ccc}
r&s&t\\
1&0&0\\
0&1&0
\end{array}\right]^{n}=\left[\begin{array}{ccc}
Q_{v,n+4}&K_{v,n+2}&tQ_{v,n+3}\\
Q_{v,n+3}&K_{v,n+1}&tQ_{v,n+2}\\
Q_{v,n+2}&K_{v,n}&tQ_{v,n+1}
\end{array}\right],$$ 
where $K_{v,n}=sQ_{v,n+1}+tQ_{v,n}$ and $r=s=t=1$. So, the following corollary can be given without proof. 

\begin{corollary}
For $n\geq 0$, we have determinant formula for Tribonacci spinors
$$\left\lbrace \begin{array}{cc} \breve{\mathrm{T}}_{n+1}\breve{\mathrm{L}}_{n+1}\mathrm{T}_{n+4}+\breve{\mathrm{T}}_{n+2}\breve{\mathrm{L}}_{n+2}\mathrm{T}_{n+2}+\breve{\mathrm{T}}_{n+3}\breve{\mathrm{L}}_{n}\mathrm{T}_{n+3}\\
\ \ -\breve{\mathrm{T}}_{n+1}\breve{\mathrm{L}}_{n+2}\mathrm{T}_{n+3}-\breve{\mathrm{T}}_{n+2}\breve{\mathrm{L}}_{n}\mathrm{T}_{4}-\breve{\mathrm{T}}_{n+3}\breve{\mathrm{L}}_{n+1}\mathrm{T}_{n+2}\end{array}\right\rbrace=4\left[\begin{array}{c} -1+i \\ 1-i\end{array}\right],$$ where $\breve{\mathrm{L}}_{n}=\breve{\mathrm{T}}_{n+1}+\breve{\mathrm{T}}_{n}$.
\end{corollary}

Now, the following corollary can be given without proof. The proof is obtained by mathematical induction.
\begin{theorem}\label{teo:5}
For every integer $n\geq 0$. The summation formula for generalized Tribonacci spinors is as follows: 
\begin{align*}
\delta \sum_{l=0}^{n}\mathrm{A}_{v,l}&=\mathrm{A}_{v,n+2}+(1-r)\mathrm{A}_{v,n+1}+t\mathrm{A}_{v,n}\\
&\ \  +\left[\begin{array}{c} (r+s)V_{3}+(r-1)V_{4}-V_{5}+i((r+s)V_{0}+(r-1)V_{1}-V_{2})\\ (r+s)V_{1}+(r-1)V_{2}-V_{3}+i((r+s)V_{2}+(r-1)V_{3}-V_{4})
\end{array}\right],
\end{align*}
where $\delta=r+s+t-1$ and $V_{n}$ is the $n$-th generalized Tribonacci number.
\end{theorem}

\begin{corollary}
For every integer $n\geq 0$. The summation formula for Tribonacci spinors is as follows: 
$$
\sum_{l=0}^{n}\mathrm{T}_{l}=\frac{1}{2}\left\lbrace \mathrm{T}_{n+2}+\mathrm{T}_{n}+\left[\begin{array}{c} -3-i\\ -2i
\end{array}\right]\right\rbrace .
$$
\end{corollary}

\section{Conclusion}

In this study, we introduced the Tribonacci spinor and generalized Tribonacci spinor by using the relationships between spinors and quaternions.Today, Tribonacci numbers are the most common generalization of Fibonacci numbers, whose applications in science are infinite. For example, in sunflower, human body, corn pyramids, Pascal triangle, representation groups. On the other hand, quaternions are defined as obtained by enlarging the set of complex numbers. Quaternions have applications in physics, mathematics, engineering and robotics. Also, by means of spinors, a shorter and simpler representation of quaternions can be obtained. By combining this information, we have obtained a new series called generalized Tribonacci spinors $\{\mathrm{A}_{v,n}\}_{n\in \mathbb{N}}$. And we think that this study, in which we have obtained a new series, will be a basic study especially for mathematicians working on geometry and algebra. In a later work, we obtain curious properties of this type of numbers.

\medskip

\end{document}